\documentclass[11pt]{amsart}

\usepackage{amssymb}
\usepackage{amsthm}
\usepackage{amsmath}
\usepackage[all]{xy}
\usepackage[margin=1in]{geometry}

\usepackage{ulem}

\usepackage[usenames]{color}

\usepackage{tikz}
\usetikzlibrary{arrows,calc}

\usepackage{tikz-cd}

\usepackage[foot]{amsaddr}

\usepackage{hyperref}

\theoremstyle{plain}

\newtheorem{thm}     {Theorem}[section]
\newtheorem{prop}    [thm]{Proposition}
\newtheorem{definition}  [thm]{Definition}
\newtheorem{cor}     [thm]{Corollary}
\newtheorem{lemma}   [thm]{Lemma}
\newtheorem{example} [thm]{Example}

\newcommand{\B}{\mathbb B}
\newcommand{\C}{\mathbb C}
\newcommand{\D}{\mathbb D}

\newcommand{\R}{\mathbb R}

\makeatletter
\@namedef{subjclassname@2020}{\textup{2020} Mathematics Subject Classification}
\makeatother

\begin{document}

\title{Kobayashi hyperbolicity in Riemannian manifolds}

\author[H. Gaussier]{Herv\'e Gaussier$^1$}
\address{H. Gaussier: Univ. Grenoble Alpes, CNRS, IF, F-38000 Grenoble, France}
\email{herve.gaussier@univ-grenoble-alpes.fr}

\author[A. Sukhov]{Alexandre Sukhov$^2$}
\address{A. Sukhov: University  of Lille,   Laboratoire
Paul Painlev\'e,  Department of 
Mathematics, 59655 Villeneuve d'Ascq, Cedex, France, and
Institut of Mathematics with Computing Centre - Subdivision of the Ufa Research Centre of Russian
Academy of Sciences, 45008, Chernyshevsky Str. 112, Ufa, Russia.}
\email{sukhov@math.univ-lille1.fr}

\date{\today}

\thanks{$^1\,$Partially supported by ERC ALKAGE}
\thanks{$^2\,$Partially suported by Labex CEMPI}

%Institut of Mathematics with Computing Centre - Subdivision of the Ufa Research Centre of Russian
%Academy of Sciences, 45008, Chernyshevsky Str. 112, Ufa, Russia.

\begin{abstract}
 We study the boundary behavior of the Kobayashi-Royden metric and the Kobayashi hyperbolicity of domains in 
 Riemannian manifolds.  As an application, we prove a Fatou type theorem on the existence, almost everywhere, of non tangential limits for bounded conformal harmonic immersed discs. We also prove a Picard theorem for conformal harmonic discs and give some examples of Kobayashi hyperbolic Riemannian manifolds.
 \end{abstract}

\maketitle

\sloppy
\section{Introduction}
The Kobayashi metric is a classical object in Several Complex Variables that encodes some geometric properties of complex manifolds. It has been the subject of numerous studies since its introduction, with seminal works. Fundamental questions are still open, both in the contexts of compact and of non-compact manifolds. The idea to introduce the Kobayashi metric on Riemannian manifolds, using conformal harmonic discs, is due to M.Gromov \cite{Gr}. F. Forstneri\v c and D.Kalaj  \cite{Fo-Ka} and B. Drinovec-Drnov\v sek and F. Forstneri\v c \cite{Dr-Fo}
 obtained several results concerning the Kobayashi metric on the Euclidean space $\R^n$ equipped with the standard metric and developed the bases of the theory of the Kobayashi hyperbolicity in that frame.
The case of arbitrary Riemannian manifolds, for which the existence of the Kobayashi metric relies on the local existence of stationary discs, is considered in \cite{Ga-Su}. The present paper is the continuation of \cite{Ga-Su}. We study the asymptotic boundary behavior of the Kobayashi metric on domains in Riemannian manifolds. As an application, we prove the H\"older-$1/2$ extension of conformal harmonic discs attached to some submanifolds given as the zero set of minimal plurisubharmonic functions. We also study localization properties, as well as the complete Kobayashi hyperbolicity of some domains in Riemannian manifolds. As an application of these results we prove a Fatou type theorem. The classical Fatou theorem asserts that a bounded harmonic (or holomorphic) function on the unit disc $\D \subset \R^2$ admits non-tangential boundary limits 
almost everywhere (a.e.) on the unit circle. We extend the Fatou theorem to conformal harmonic maps from the unit disc to an arbitrary Riemannian manifold.
As another application we prove a Riemannian version of the Picard theorem, providing sufficient condition for a conformal harmonic map from the punctured unit disc into a Riemannian manifold to extend as a conformal harmonic map on the whole unit disc.

This paper is dedicated to Professor J.Globevnik.

%Our approach is adapted from the one followed by F. Haggui and A. Khalfallah \cite{Ha-Kh} who considered the case of almost complex manifolds. 

\section{Preliminaries}

We denote by $\D$ the unit disc in $\R^2$ (we often identify $\R^2$ with the complex plane $\C$), by $ds^2$ the standard Riemannian metric on $\R^2$ and by $dm$ the standard Lebesgue measure on $\R^2$. For every $x \in \R^2$ and every $\lambda > 0$, we set $D(x,\lambda):=\{\zeta \in \C /\ |\zeta-x| < \lambda\}$. Throughout the paper, we consider a Riemannian manifold $(M,g)$, where $M$ is a smooth $C^{\infty}$ real manifold of real dimension $n \geq 2$ and $g$ is a smooth Riemannian metric of class $C^{\infty}$ on $M$. We denote by $d_g$ the distance induced by $g$, defined as the infimum of the length of $C^1$ paths joining two points. For every $p \in M$ and every $r > 0$, we denote by $B_g(p,r)$ the ball $B_g(p,r) := \{q \in M /\ d_g(p,q) < r\}$.

%%%%%%%%%%%%%%%%%%%%%%%%%%%%%%%%%%%%%%%%%%%%

\subsection{Minimal surfaces, conformal and harmonic maps}

In this Subsection, we present a precise version of Lemma~2.3 Part (i) in \cite{Ga-Su} ; this is the content of Lemma~\ref{LemNW2}. For the sake of completeness we give a complete proof of Lemma~\ref{LemNW2} and we recall the  notions necessary for its undestanding, although these are essentially exposed in \cite{Ga-Su}.

%We assume that all structures are smooth of class $C^\infty$. We denote by $dist_g$ the distance induced by $g$, defined as the infimum of the length of $C^1$ paths joining two points. For every $p \in M$ and every $r > 0$, let $\B(p,r):=\{q \in M \vert \ dist_g(p,q) < r\}$.

Let $(M,g)$ be a Riemannian manifold.  A map $u: \D \to M$ is called harmonic if it is a critical point of the energy integral

\begin{eqnarray*}
E(u) = \int_\D \vert du \vert^2dm.
\end{eqnarray*}

A harmonic map satisfies the Euler-Lagrange  equations. 
 %\begin{eqnarray}
%\label{EL}
%\Delta_g u + B(u) (du, du) = 0.
%\end{eqnarray}
Using the complex coordinates $\zeta=\xi + i \eta$ %$z = x + iy$
on $\D$ and local coordinates on $M$ near $u(\D)$, one can write the Euler-Lagrange equations  in the form 

\begin{eqnarray}\label{harm-eq}
\label{EL2}
\Delta u^i + (\Gamma^{i}_{jk} \circ u) \frac{\partial u^j}{\partial \zeta}\frac{\partial u^k}{\partial \overline \zeta} = 0
\end{eqnarray}
(see \cite{Jo2}). Here $\displaystyle \Delta = 4\frac{\partial^2}{\partial \zeta \partial \overline{\zeta}}$ is the standard Laplace operator, $\Gamma^{i}_{jk}\ (i, j, k = 1, \dots n)$ denote the Christoffel symbols and $u=(u^1,\dots,u^n)$. In case where the local  coordinates are normal at $p$ for the Levi-Civita connection associated to the metric $g$, we have $g_{ij}(p) = \delta_{ij}$
(the Kronecker symbols), $\partial g_{ij}/\partial x_k(p) = 0$ and $\Gamma^k_{ij}(p) = 0$ for every $i,j,k = 1,\dots,n$.

%For every $p \in M$, let $g_p$ denote the Euclidean metric on $T_pM$. Then there exists $r_p > 0$ such that the exponential map
%$$
%\exp_p : B_{g_p}(0,r_p) \subset T_pM \rightarrow \exp_p(B_{g_p}(0,r_p)) \subset M
%$$
%is a smooth $C^{\infty}$ diffeomorphism.

%Let $(e_1, \dots, e_n)$ be an orthonormal basis of $T_pM$ and consider the diffeomorphism
%$$
%\begin{array}{ccccc}
%\Phi & : & B_{g_p}(0,r_p) & \longrightarrow & \exp_p(B_{g_p}(0,r_p))\\
%& & x=(x_1,\dots,x_n) & \longmapsto & \exp_p(\sum_{i=1}^n x_i  e_i).
%\end{array}
%$$
%For every $q \in \exp_p(B_{g_p}(0,r_p))$, the coordinates $(x_1(q), \dots, x_n(q))$ defined by $q=\Phi(x_1(q), \dots, x_n(q))$ are called normal coordinates at $p$, and $x(p) = 0$.

The Euler-Lagrange equations form a second order elliptic quasilinear PDE system. The initial regularity of $u$ may be prescribed in H\"older or Sobolev spaces. It follows by elliptic regularity that a solution $u$ of (\ref{EL2}) is a smooth $C^{\infty}$ map on $\D$.

A smooth map $u : \D \to M$ is called conformal if  the pull-back $u^*g$ is a metric conformal to $ds^2$ i.e.,
there exists a smooth function $\phi$ such that $u^*g = e^\phi ds^2$ on $\D$. %Recall that  any Riemannian metric $h$ on $\D$ admits conformal coordinates. This means that there exists a smooth diffeomorphism $\Phi: \D \to \D$, depending on $h$, such that $\Phi^*h$ is conformal to $ds^2$. %Note that without additional assumptions, this result is true only for manifolds of real dimension 2.
%If $u: \D \to M$ is a smooth immersion, we take $h = u^*g$, and the composition $u \circ \Phi$ becomes a conformal mapping. 
%Of course, $\Phi$ depends on $u$ and is not unique. Using the group of conformal automorphisms of $\D$, we can always achieve the conditions $\Phi(0) = 0$ and $\Phi(1) = 1$. 
A surface in $(M,g)$ is called minimal if its mean curvature induced by $g$ vanishes. A conformal  immersion (i.e., its image) is minimal if and only if 
it is harmonic (see \cite{Jo2}).
%The energy functional has very important compactness properties in suitable functional spaces. This makes it a useful tool in order to study boundary values problems for minimal surfaces, in particular, the Plateau problem. However, in some cases it is more convenient to work with the area functional. Here we follow the approach of B. White \cite{Wh1}.

Let $u:\D \longrightarrow M$ be a smooth immersion. We denote by $g_{\mathbb D}:=u^*(g)$ the Riemannian metric on $\mathbb D$, pullback of the metric $g$ by $u$. Let $G_{\mathbb D}$ be the matrix $(G_{\mathbb D})_{i,j} = g_{\mathbb D}(\partial/\partial x_i,\partial/\partial x_j)$, for $i, j =1,2$, where for convenience we write $x_1=\xi,\ x_2=\eta$ in the matrix notation. %We also denote $u_x:=u_{_*}\left(\partial / \partial x\right )$ and $u_y:=u_{_*}\left(\partial / \partial y\right)$. %In particular, the scalar product of $\partial / \partial x_i, \partial / \partial x_j$ evaluated at $(x,y) \in \mathbb D$ is equal to $g_{\mathbb D}(\partial / \partial x_i, \partial / \partial x_j) = g_{u(x,y)}(u_{x_i},u_{x_j})$. For convenience, we just write $g_{u(x,y)}(u_{x_i},u_{x_j})=:g_{u}(u_{x_i},u_{x_j})$.
The area functional $A(u)$ of an immersion $u$ is defined by

\begin{eqnarray}
A(u) = \int_\D (\det G_{\mathbb D})^{1/2}dm.
\end{eqnarray}

One may view $A$ as a real map defined on the space of smooth immersions. A smooth immersion $u$ is called  stationary if the differential $DA$ of $A$ vanishes at $u$ i.e., $DA(u) = 0$. As it is shown in \cite{Wh1}, an immersion is stationary if and only if its image is a minimal surface. Therefore, after a suitable reparametrisation of the unit disc, a stationary immersion  becomes a conformal harmonic map.

In the rest of the paper, we refer conformal harmonic immersions from $\D$ to $M$ as conformal harmonic immersed discs. Similarly, we  refer indifferently to stationary or minimal discs. Note that 
we consider stationary discs with arbitrary parametrizations, not necessarily the conformal harmonic ones. Note also that we sometimes identify an immersed disc with its image when this does not lead to any confusion.

According to \cite{Wh1}, an immersed disc is stationary if $\frac{d}{dt}_{|t=0} A(\varphi^t) = 0$ for any one-parameter family of transverse deformations of $u(\mathbb D)$ and
$$
\frac{d}{dt}_{|t=0} A(\varphi^t) = 0 \Leftrightarrow \int_{\mathbb D} \left(\frac{d}{dt}_{|t=0}\varphi^t \right) \cdot H(u) dm = 0,
$$
for some quasilinear elliptic operator $H$ such that $H(u)$ is normal to $u(\mathbb D)$. %Let $\pi : \D \rightarrow TM$ be a smooth map such that for each $\zeta \in \D$, $\pi(\zeta)$ is a projection onto a distribution of real subspaces tranverse to $u(\D)$.sousespaces transverses a u(D).the conormal bundle $T_{u(\zeta)}^{\perp}$ of $u(\D)$.
Then $DA(u) = 0$ if

\begin{equation}\label{stat-eq}
\pi((H(u)) = 0.
\end{equation}
where $\pi$ denotes a projection to a distribution of subspaces of $TM$ transverse to $u(\D)$. The projection $\pi$ is defined on $T\mathcal U$ where $\mathcal U$ is a sufficiently small neighborhood of $u(\D)$.
%sur une distribution des sousespaces transverses a u(D)along some smooth nonvanishing section of $TM$ near $u(\mathbb D)$, transverse to $u(\mathbb D)$.

In suitable local coordinates on $M$ we consider an immersed disc $u : \mathbb D \rightarrow \mathbb R^{n}$ given as a graph
$$
u(\xi,\eta) = (\xi,\eta,u^3(\xi,\eta), \dots, u^n(\xi,\eta)),
$$
where $u^3, \dots, u^n$ are smooth $C^{\infty}$ functions. %Then
%$$
%u_x:=(1,0,u^3_x, \dots, u^n_x)
%$$
%and
%$$
%u_y:=(1,0,u^3_y, \dots, u^n_y).
%$$
A one-parameter family $\varphi^t$ of transverse deformations of $u$ can be expressed in these coordinates as $\varphi^t(\xi,\eta)=(\xi,\eta,u^3(\xi,\eta)+th^3(\xi,\eta), \dots, u^n(\xi,\eta)+th^n(\xi,\eta))$ with $h^3, \dots, h^n \in C^{\infty}(\mathbb D,\mathbb R)$. %Then $u$ is stationary if and only if
Then, we can take for $\pi$ the orthogonal projection on $\{0\} \times \{0\} \times \mathbb R^{n-2}$ and $\pi(H(u)) = 0$ (i.e. $u$ is stationary)  if and only if
\begin{equation}\label{stat2-eq}
H^j(u) = 0, \ j = 3,...,n.
\end{equation}

Moreover, $H^j(u) = u^j_{\xi \xi}+u^j_{\eta \eta}+\psi^j(u_\xi,u_\eta,u_{\xi \xi},u_{\xi \eta},u_{\eta \eta}) \ {\rm for \ } j = 3,\dots, n$, where $\psi^j$ is a smooth $C^{\infty}$ function, without constant or linear terms with respect to $u_\xi,u_\eta,u_{\xi \xi},u_{\xi \eta},u_{\eta \eta}$. In particular, the operator $\pi(H(u))$ is a quasilinear elliptic operator whose linearization at $u^0: (\xi,\eta) \in \D \mapsto (\xi,\eta,0,\dots,0) \in \R^n$ is

\begin{eqnarray}\label{linear-eq}
\left(0, 0, \Delta u^3, \cdots, \Delta u^n\right).
\end{eqnarray}

We denote $e_1:=(1,0) \in \R^2$. The following result claims the existence of a conformal harmonic disc with prescribed $1$-jet.

\begin{lemma}
\label{LemNW2}
 Let $p \in M$ and let $E$ be a 2-dimensional subspace of  $T_pM$. Then there exists a conformal harmonic immersion 
$u: \D \to M$ with $u(0) = p$ and such that the tangent space $T_p u(\Delta)$ coincides with $E$. Furthermore, this immersion depends smoothly on  $p$, $E$ and $g$.
\end{lemma}
\begin{proof} We consider normal coordinates $(x_1,\dots,x_n)$ on $M$ in which $p = 0$. %We also assume that  these coordinates are normal for the Levi-Civita connexion of $(M,g)$, that is $g_{ij}(0) = \delta_{ij}$, $\partial g_{ij}/\partial x_k(0) = 0$ and $\Gamma^k_{ij}(0) = 0$ for every $i,j,k = 1,\dots,n$.  %Then in these coordinates $g_{ij}(0) = \delta_{ij}$ and the first order partial derivatives of $g_{ij}$ vanish at $0$.
Consider the metric $h_t(x_1,\dots,x_n) = t^{-2}g_t(x_1,\dots,x_n) := t^{-2}g(tx_1, \dots, tx_n)$ for $t > 0$. %We notice that such a metric is isometric to $g$.
The metric $h_t$ has the same set of stationary surfaces as $g_t$ and $g$. Moreover, the metric $h_t$ converges to the standard metric $g_{st}$ of $\R^n$ in any $C^k$-norm on any compact subset of  $\R^n$, as $t \rightarrow 0$.

We may assume that $E$ is generated by the vectors $(1,0,...,0)$ and $(0,1,...,0)$. We search for a suitable stationary immersion of the unit disc of the form $(\xi,\eta) \mapsto (\xi,\eta,f(\xi,\eta))$ i.e., as the graph of a vector function $f: \D \to \R^{n-2}$. Consider the system (\ref{stat2-eq}) for 
the metric $h_t$.  Its linearization at $f=0$ is given by (\ref{linear-eq}) which is a surjective operator  from $C^{2,\alpha}(\D, \R^{n-2})$ to $C^{0,\alpha}(\D,\mathbb R^{n-2})$, with any fixed $\alpha \in (0,1)$ (this follows from the regularity of the Laplace potential, see more details in \cite{Ga-Su}). %Indeed, this follows from the stated above  regularity property of the Laplace operator in the Holder scale.
By the implicit function theorem, the system~(\ref{stat2-eq}) admits solutions for $t > 0$ small enough. Namely, for every sufficiently small $t$, there exists a minimal surface, given by a smooth stationary immersion $u(t,p,E) : \D \rightarrow \mathbb R^n$, for the Riemannian metric $h_t$, such that  $u(t,p,E) : (\xi,\eta) \in \D \mapsto  (\xi,\eta,f_{t,p,E}(\xi,\eta))$ depends smoothly on $(t,p,E)$ and is a small deformation of the map $(\xi,\eta) \mapsto (\xi,\eta,0)$. In particular, $u(t,p,E)(0)$ is close to $p$ and $du(t,p,E)(0) (T_0 \D)$ is close to $E$. Now, if $\mathcal U_{p}$ is a small neighborhood of $p$ in $\R^n$ and $\mathcal V$ is a small neighborhood of $E$ in the  Grassmanian $Gr(2,n)$, the same reasoning implies that for every sufficiently small $t$, the set $(u(t,\tilde p,\tilde E)(0), du(t,\tilde p,\tilde E)(0) (T_0\D))$ where $\tilde p \in \mathcal U_{p}$ and $\tilde E \in \mathcal V$, fills an open neighborhood of $(p,E)$ in $\R^n \times  Gr(2,n)$. Hence, for sufficiently small $t$, there exists $(\tilde p,\tilde E) \in \mathcal U_{p} \times \mathcal V$ such that $u(t,\tilde p,\tilde E) (0) = p$ and $du(t,\tilde p,\tilde E) (0) (T_0 \D)= E$.  Note that all solutions are $C^\infty$ smooth by the elliptic regularity. Finally, being a solution of the system~(\ref{stat2-eq}), $u(t,\tilde p,\tilde E) : \D \rightarrow \R^n$ is a stationary disc for $g_t$ and therefore for $g$. Hence, this disc  becomes a  conformal harmonic immersion  for $g$ after a suitable reparametrization.  The smooth dependence on $p$, $E$ and $g$ follows from the implicit function theorem. This proves Lemma~\ref{LemNW2} . 
\end{proof}

\subsection{MPSH functions}

%Throughout the paper, we consider a Riemannian manifold $(M,g)$, where $M$ is a smooth real manifold of real dimension $n \geq 2$ and $g$ is a smooth Riemannian metric of class $C^{\infty}$ on $M$. We denote by $d_g$ the integrated distance induced by the metric $g$. When the metric $g$ is fixed, we sometimes write $d$ instead of $d_g$.

%\subsection{Normal coordinates and the Hessian operator}

%For every $p \in M$ and every $r > 0$ we denote by $B_g(p,r)$ the ball
%$$
%B_g(p,r) := \{q \in M /\ d_g(p,q) < r\}.
%$$
%Let $(M,g)$ be a Riemannian manifold.

%Let $\nabla$ be the Levi-Civita connection on $(M,g)$. This is the unique affine connection on $TM$ without torsion and such that the Riemannian metric $g$ is parallel. %Let $p \in M$ and let $x=(x_1, \dots, x_n)$ be local coordinates at $p$. We denote by $\Gamma^k_{ij}$ the Christoffel symbols of the Levi-Civita connection in the $x$ coordinates (
%Let $\rho$ be a function of class $C^2$ on $M$. We recall that the {\sl Hessian} of $\rho$ is the symmetric bilinear form $\nabla^2 \varphi$ defined on $TM \times TM$ by
%$$
%\left(\nabla^2 \rho \right)(X,Y) = X\left(Y(\rho) \right)-\left(\nabla_X Y\right)(\rho).
%$$

%Let $\rho$ be a smooth $C^2$ function defined on a (sufficiently) small neighborhood of $p$.
%In local coordinates at a point $p \in M$, the Hessian of $\rho$ is given by

%\begin{equation}\label{hessian-eq}
%\nabla^2 \rho = \sum_{i,j=1}^n \left(\frac{\partial^2 \rho}{\partial x_i \partial x_j} - \sum_{k=1}^n \Gamma^k_{ij} \frac{\partial \rho}{\partial x_k}\right)dx_i \otimes dx_j.
%\end{equation}
 
%\subsection{MPSH functions}

Similarly to  \cite{Dr-Fo}, we introduce the following

\begin{definition}\label{MPSH-def}
Let $\rho : M \rightarrow [-\infty,+\infty)$ be an upper semi-continuous function. We say that $\rho$ is a minimal plurisubharmonic (MPSH) function if for every conformal harmonic immersed disc $u : \D \rightarrow M$, the composition $\rho \circ u$ is a subharmonic function on $\D$.
\end{definition}

In the case of  $C^2$ functions, we have the following characterization of minimal plurisubharmonicity

\begin{lemma}\label{MPSH-lem}
Let $\rho : M \rightarrow \mathbb R$ be a function of class $C^2$. Then $\rho$ is MPSH if and only if $\Delta (\rho \circ u) \geq 0$ for every conformal harmonic immersed disc $u : \D \rightarrow M$.
\end{lemma}

%\begin{proof}
%By definition.
%\end{proof}

Since a conformal harmonic disc has only a finite number of singular points on compact subsets of $\D$,  a function $\rho : M \rightarrow \mathbb R$ of class $C^2$ is MPSH if and only if $\Delta (\rho \circ u) \geq 0$ for every conformal harmonic disc $u : \D \rightarrow M$. Moreover, there are sufficiently many conformal harmonic immersed discs for the MPSH notion to be consistent, as showed by Lemma~\ref{LemNW2}.
%Indeed, we have the following refined version of Lemma~2.3 in \cite{Ga-Su}:

%\begin{lemma}
%\label{LemNW}
% Let $p \in M$ and let $E$ be a 2-dimensional subspace of  $T_pM$. Then there exists a conformal harmonic immersion  $u: \D \to M$ with $u(0) = p$ and such that the tangent space $T_p u(\D)$ conicides with $E$. Furthermore, this immersion depends smoothly on  $p$ and $E$.
%\end{lemma}

%Since the proof is a direct adaptation of the one in \cite{Ga-Su}, we briefly present  it in the Appendix.

\begin{lemma}
\label{LemPSH1}
Let $\rho$ be a $C^2$ function on $M$ and $u: \D \to M$ be a conformal harmonic disc. Then $\Delta (\rho \circ u)(0)$ depends only on the $1$-jet of $u$ at the origin.
\end{lemma}

\begin{proof}
%We recall that the {\sl Hessian} of $\rho$ is the symmetric bilinear $2$-tensor field $\nabla^2 \rho$ defined on $TM \times TM$ by $
%\left(\nabla^2 \rho \right)(X,Y) = X\left(Y(\rho) \right)-\left(\nabla_X Y\right)(\rho)$.
%Let $\rho$ be a smooth $C^2$ function defined on a (sufficiently) small neighborhood of $p$.
%In local coordinates at a point $p \in M$, the Hessian of $\rho$ is given by

%\begin{equation}\label{hessian-eq}
%%\nabla^2 \rho = \sum_{i,j=1}^n \left(\frac{\partial^2 \rho}{\partial x_i \partial x_j} - \sum_{k=1}^n \Gamma^k_{ij} \frac{\partial \rho}{\partial x_k}\right)dx_i \otimes dx_j.
%\end{equation}
%We set $\partial / \partial x = (1,0)$, $\partial / \partial y = (0,1)$, $p=u(0)$. 

Let $(x_1, \dots, x_n)$ be local coordinates  a neighborhood $U$ of $p$,  and let $u=(u^1,\dots,u^n): \D \longrightarrow U$ be a conformal harmonic disc. We still use the notation $\zeta = \xi + i \eta$ for coordinates in $\D$.
Since $u$ is harmonic, it follows from (\ref{harm-eq}) that 
\begin{eqnarray}
\label{Hess1}
\sum_{i=1}^n \Delta u^i = - \sum_{i,j,k=1}^n \left(\Gamma^i_{j,k} \circ u\right) \frac{\partial u^j}{\partial \zeta} \frac{\partial u^k}{\partial \overline \zeta}.
\end{eqnarray}
Hence, we obtain after a direct computation:

\begin{eqnarray}\label{eq2}
\ \ \Delta (\rho \circ u) =  \sum_{ij}\left(\frac{\partial^2 \rho}{\partial x_i \partial x_j} \circ u\right)
  \left ( \frac{\partial u^i}{\partial \xi} \frac{\partial u^j}{\partial \xi} +\frac{\partial u^i}{\partial \eta} \frac{\partial u^j}{\partial \eta} \right )
    - \sum_{i,j,k}  \left(\frac{\partial \rho}{\partial x_i} \circ u\right)\left(\Gamma^i_{j,k} \circ u\right) \frac{\partial u^j}{\partial \zeta} \frac{\partial u^k}{\partial \overline \zeta}.
\end{eqnarray}
%for each $\zeta \in \D$. %From this Lemma follows.
\end{proof}

Note  that if $(x_1,...,x_n)$ are normal coordinates at $p = 0$, then $\Gamma^j_{k,l}(p) = 0$. If $u(0) = p$,  then (\ref{eq2}) becomes at $p$:
\begin{eqnarray}
\label{eq3}
 \Delta (\rho \circ u)(0)  =  \sum_{ij}\frac{\partial^2 \rho}{\partial x_i \partial x_j}(p)
  \left ( \frac{\partial u^i}{\partial \xi} \frac{\partial u^j}{\partial \xi} +\frac{\partial u^i}{\partial \eta} \frac{\partial u^j}{\partial \eta} \right )(0).
    \end{eqnarray}

Finally we have the following
\begin{lemma}
\label{LemPSH2}
Let $\rho$ be a $C^2$ function on $M$. Assume that for every $p$ in $M$ and every 2-dimensional subspace $E$ of $T_pM$, there exists a conformal harmonic disc $u$ such that $u(0) = p$ and $T_p u(\D) = E$, and such that $\Delta (\rho \circ u)(0) \ge 0$. Then $\rho$ is an MPSH function.
\end{lemma}
\begin{proof} Let $f : \D \to  M$ be a conformal harmonic immersed disc. Let $p = f(\zeta_0)$ for some $\zeta_0 \in \D$. By assumption, there exists a conformal  harmonic  disc $u:\D \to M$ satisfying $u(0) = p$, $T_pf(\D) = T_pu(\D)$, and $\Delta (\rho \circ u)(0) \ge 0$. Let $h=(h_1,h_2)$ be a holomorphic automorphism of $\D$ such that $h(0)=\zeta_0$. Then $g := f \circ h$ is a conformal harmonic immersed disc, with $g(0)=p$ and $T_pg(\D) = T_pf(\D)$. Moreover

$$
\displaystyle \Delta(\rho \circ g)(0) = \Delta(\rho \circ f)(\zeta_0) \left(\left| \frac{\partial h_1}{\partial \xi}(0)\right|^2 + \left| \frac{\partial h_2}{\partial \xi}(0)\right|^2\right).
$$

Denote by $(e_1,e_2)$ the canonical basis of $\R^2$ and consider the vectors $v_j:=du(0)(e_j)$, $j=1,2$. Since the map $g$ is conformal, there exists an orthonormal basis 
$a_1,a_2$ in  $\R^2$  and $\lambda \in \R$  such that $dg(0)(a_1) = \lambda v_1$, $dg(0)(a_2) = \lambda v_2$. Let $L$ be an orthogonal tranformation  of $\R^2$ such that $L( e_j) = a_j$, $j=1,2$. Then $\tilde g:=g \circ L$ is a conformal harmonic immersed disc satisfying $\tilde g(0) = p$, $d \tilde g(0)(e_j) = \lambda v_j$, $j= 1,2$ and by Lemma~\ref{LemPSH1} (see, for example (\ref{eq3}):

$$
\lambda^2 \Delta (\rho \circ u)(0) = \Delta (\rho \circ \tilde g)(0) = \Delta (\rho \circ g)(0).
$$

Hence $\Delta (\rho \circ g)(0) \geq 0$ and so $\Delta (\rho \circ f)(\zeta_0) \geq 0$.
\end{proof}

%\vspace{2mm}
%\subsection{Some examples}

%Let $p \in M$ and let $\rho$ be a smooth real $C^2$ function defined in a neighborhood $U$ of $p$. %If $\rho$ is strictly convex on $U$ i.e., $\rho$ is of class $C^2$ and $\nabla^2 \rho$ is positive definite on $U$, then it follows from (\ref{eq2}) that $\rho$ is MPSH on $U$. In particular, we have
Consider the following basic example. Let $x=(x_1,\dots,x_n)$ be normal coordinates at $p \in M$, in particular  $x(p) = 0$. Then the function $|x|^2 : q \mapsto |x(q)|^2=\sum_{j=1}^nx_j^2(q)$  is defined and MPSH in a neighborhood $U$ of $p$. Indeed, as we saw in the proof of Lemma \ref{LemNW2}, locally there exists a smooth 1-parameter family of Riemannian metrics (denoted by $h_t$ in the proof of Lemma \ref{LemNW2}) such that that for every $t> 0$ the set of conformal harmonic maps with respect to $h_t$ is the same as the one for the initial metric $g$ and, furthermore, $h_0 = g_{st}$. The function is clearly MPSH for $g_{st}$ and there exists a constant $C > 0$ such that for every linear conformal disc $u$ through the origin one has $(\Delta \rho \circ u)(0) \ge C$. Then by continuity the same holds for conformal harmonic discs with respect to the metric $h_t$ for $t$ small enough. 
Then we can apply Lemma \ref{LemPSH2}.

%Indeed, using (\ref{hessian-eq}) we have
%$$
%\nabla^2 (|x|^2) = \sum_{i=1}^n \left(1 + \mathcal O(|x|^2)\right) dx_i \otimes dx_i + \sum_{1 \leq i < j \leq n}\mathcal O(|x|^2) dx_i \otimes dx_j,
%$$
%which is a small deformation of $\sum_{i=1}^n dx_i \otimes dx_i$ near $p$. Hence $|x|^2$ is strictly convex near $p$.

\vspace{2mm}
A function $\rho: M \to \R$ is called strictly MPSH if for every point $p \in M$ there exist local normal coordinates $x$ centered at $p$ and $\varepsilon > 0$ such that $\rho - \varepsilon \vert x \vert^2$ is an MPSH function in a neighborhood of $p = 0$. Clearly, such a positive constant $\varepsilon$ can be choosen independently of $p \in K$, for every compact subset $K \in \Omega$. Notice that a strictly convex funtion is strictly MPSH. We have the following

\begin{lemma}
\label{LemLog}
Suppose that the normal coordinates $x$ are chosen at the point $p = 0$.
Then the function $\phi(x) = \log \vert x \vert + A \vert x \vert$ is strictly MPSH in a neighborhood of the origin for a sufficiently large $A > 0$. In particular, for every $q$ sufficiently close to $p$ the function $ \phi_q(x) = \log \vert x - x(q) \vert + A \vert x - x(q) \vert$ is MPSH in the same neighborhood of the origin.
\end{lemma}
\proof  By Lemma \ref{LemNW2}, there exists a family of immersed minimal discs centered at the origin 
such that their tangent spaces at the origin fill the Grassmannian $Gr(2,n)$.  Furthermore, these discs depend smoothly on the spaces that run over $Gr(2,n)$.

We denote by $\zeta=\xi + i \eta$ the coordinates in $\D$. Let $u: \D \longrightarrow (\R^n,g)$ be a conformal harmonic immersion with $u(\zeta) = (\xi,\eta,0,...,0) + O(\vert \zeta \vert^2)$. Then 
$$
\Delta \log |u|(\zeta) = O\left(1/|\zeta| \right)
$$
and
$$
\Delta \vert u\vert(\zeta) = \left(1/|\zeta|\right)\left(1 + O(|\zeta|)\right).
$$

This implies that $\log |u| + A \vert u \vert$ is subharmonic near the origin for a sufficiently large constant $A>0$. Notice that up to now the constant $A > 0$ depends on $u$.
Hence, there exists $B > 0$ such that 
\begin{eqnarray}
\label{lap1}
\Delta \phi(u(\zeta)) \ge B / \vert \zeta \vert
\end{eqnarray}
 near the origin in $\D$. Since 
everything depends smoothly on parameters, the same estimate holds for minimal discs of the above family close enough to the 
initial disc $u$. The tangent spaces of these discs fill an open non-empty subset of $Gr(2,n)$. By compactness of $Gr(2,n)$, we obtain that 
the estimate (\ref{lap1}) holds for all discs with uniform $A,B > 0$. The above family of minimal discs also smoothly depends on it 
starting point, the origin in our case. Moving (using Lemma \ref{LemNW2}) this family to a point $q$ close enough to the origin, we obtain by continuity
that the estimate (\ref{lap1}) holds for all minimal discs of this family, with uniform $B > 0$. This implies the first part of Lemma~\ref{LemLog} in view of Lemma \ref{LemPSH2}. The second part follows from the fact that for $q$ sufficiently close to $p$, we can choose normal coordinates $x^q$ at $q$, defined on $U$ (restricting $U$ if necessary), that are small smooth deformations of $x$ on $U$. Then the function $\log \vert x^q \vert + A \vert x^q \vert - \varepsilon |x^q|^2$ is MPSH on $U$ for some positive $\varepsilon$ not depending on $q$. Finally the coordinates $x-x(q)$ centered at $q$, defined on $U$, are small smooth deformations of the normal coordinates $x^q$ on $U$.

\qed

\section{Localization principle and preliminary boundary estimates}
Let $(M,g)$ be a Riemannian manifold of dimension $n$ and  let $D$ be a domain in $M$. We denote by $F_D(x,\xi)$ the value of the Kobayashi-Royden pseudometric at $(x,\xi) \in TM$. The definition of the Kobayashi-Royden metric is similar to that of the classical case of complex manifolds, replacing holomorphic discs with conformal harmonic immersed discs, see \cite{Ga-Su}. Denote by $\B^n$ the unit ball of $\R^n$. The main result of this section is the following localization principle for the Kobayashi-Royden metric. We follow the approach of  \cite{Ch-Co-Su} developed  for the case of 
complex  manifolds.

\begin{thm}\label{Kobest-thm}
\label{LocTheo}
Let $x: U \to 3\B^n$ be a normal coordinate neighborhood in $M$ centered at a point $p \in D$ (in particular, $x(p) = 0$). Assume that $U$ is small enough, in particular, $\vert x \vert^2$ is an MPSH function on $U$. Let $\rho$ be a negative MPSH function  in $D$ such that the function $\rho - \varepsilon \vert x \vert^2$ is MPSH on $D \cap U$ and $\vert \rho \vert \le B$ in $D \cap x^{-1}(2\B^n)$ for some constants $\varepsilon, B > 0$.
Then there exists a constant $C=C(\varepsilon,B) > 0$, independent of $\rho$, such that for every $w \in D \cap x^{-1}(\B^n)$ and every tangent vector $\xi \in T_w(M)$:
\begin{eqnarray*}
F_D(w,\xi) \ge C \vert \xi \vert \vert \rho(w) \vert^{-1/2}.
\end{eqnarray*}
\end{thm}
%\sout{Here we use the notation $\vert \cdot \vert$ for the norm induced on $U$ by the local coordinates $x$ and the Euclidean metric.}

The coordinate neighborhood $U$ is not assumed to be contained in $D$. Therefore, this result gives a first (non optimal) asymptotic behavior estimate of $F_D$ near the boundary of $D$ in $M$. Note also that no conditions such as boundedness or Kobayashi hyperbolicity are imposed on $D$.

\vspace{2mm}
\noindent{\sl Proof of Theorem~\ref{LocTheo}.} The proof consists of several steps.

\bigskip

{\bf Step 1. Construction of suitable MPSH functions.}
Consider a smooth nondecreasing function $\psi$ on $\R_+$ satisfying $\psi(t) = t$ for $0 \le t \le 1/2$ and $\psi(t) = 1$ for $t \ge 3/4$. For each point $q \in M$ satisfying $\vert x(q) \vert < 2$, we define the function 
$\Psi_q =  \psi (\vert x - x(q) \vert^2) \exp(A\psi(\vert x - x(q) \vert)) \exp(\lambda \rho)$ on $D \cap U$, and 
$\Psi_q = \exp(A\psi(\vert x - x(q)\vert^2)) \exp(\lambda \rho)$ on $D \setminus U$;  the positive constants $A$ and  $\lambda$ will be chosen later.

The function $\log \Psi_q (x) = \log \psi (\vert x - x(q) \vert^2) + A \psi(\vert x - x(q) \vert) + \lambda \rho$  is MPSH on $D \setminus \{ \vert x - x(q) \vert^2 \le 3/4 \}$.  On the other hand, according to Lemma~\ref{LemLog} there exists a constant $A > 0$ such that the function $\log \psi(\vert x - x(q) \vert^2) + A \vert x - x(q) \vert + A \vert x \vert^2$ is MPSH on $U$.  Moreover, by assumption the function $\rho - \varepsilon \vert x \vert^2$ is MPSH on $D \cap \{ \vert x - x(q) \vert \le 1 \}$. Hence, taking $\lambda = A/\varepsilon$, we obtain that the function $\log \Psi_q$ is MPSH on $D \cap \{ \vert x - x(q) \vert \le 1 \}$ and, therefore, everywhere on $D$.

\bigskip

{\bf Step 2. Preliminary  estimate of the Kobayashi-Royden metric.}
Let $u : \D \to D$ be a conformal harmonic mapping satisfying $u(0) = q$ with $\vert x(q) \vert < 2$. Then the function $\phi(\zeta) = \Psi_q(u(\zeta))/\vert \zeta \vert^2$ is subharmonic on the punctured disc $\D \setminus \{ 0 \}$ and is upper bounded by  $\exp(A)$ as $\zeta$ tends to the unit circle.

Without loss of generality, we assume that the local coordinates are normal at $q$. Since $u$ is a conformal map, it is easy to see that $\lim_{\zeta \to 0} \phi(\zeta) = \vert du(0)e_1) \vert^2\exp (A\rho(q)/\varepsilon)$, where $e_1 = (1,0)$. Hence, $\phi$ extends on $\D$ as a subharmonic function. By the maximum principle for subharmonic functions, we have $\vert du(0)e_1) \vert^2 \le \exp(A)\exp(-A\rho(q)/\varepsilon)$. Now by the definition of the Kobayashi-Royden metric, we obtain the estimate

\begin{eqnarray}
\label{est1}
F_D(q,\xi) \ge \exp((-A+A \rho(q))/2\varepsilon) \vert \xi \vert \ge N(\varepsilon,B) \vert \xi \vert
\end{eqnarray}
for any $q \in D \cap x^{-1}(2\B)$ and $\xi \in T_qD$. Here $N(\varepsilon, B) = \exp((-A - AB)/2\varepsilon)$.

\bigskip

{\bf Step 3. Localization of the Kobayashi balls on $D$. }
As it is shown in \cite{Ga-Su}, $F_D$ is an upper semi-continuous function on the tangent bundle $T_D$, and the Kobayashi distance  $d_D$ of $D$ is the integrated form of $F_D$ i.e., for any points $p,q \in D$ we have

\begin{eqnarray}
\label{Roy}
d_D(p,q) = \inf_{\gamma \in \Gamma(p,q)} \int_0^1 F_D(\gamma(t), d\gamma(t))dt
\end{eqnarray}
where the infimum is taken over the set $\Gamma(p,q)$ of all $C^1$- smooth paths $\gamma: [0,1] \longrightarrow D$ with $\gamma(0) = p$, $\gamma(1) = q$. Denote by $B_D(q,\delta)$ the Kobayashi ball with respect to $d_D$, centered at $q \in D$ and of radius $\delta > 0$. We want to compare $B_D(q,\delta)$ with a suitable ball with respect to the Euclidean ball (in local coordinates). This in turn allows to control a distortion of conformal harmonic discs giving the desired localization of the Kobayashi-Royden pseudometric.

\begin{lemma}
\label{LocLemma}
For any point $q $ in $D \cap x^{-1}(\B^n)$ and any $\delta \le N=N(\varepsilon,B)$, the Kobayashi ball $B_D(q,\delta)$ is contained in $D \cap \{ \vert x - x(q) \vert < \delta/N \}$.
\end{lemma}
\proof Fix a point $w \in D$. Setting $G = \{ \tilde w \in U: \vert x(\tilde w)  - x(q) \vert  < 1 \}$, we obtain from (\ref{est1}) and (\ref{Roy}) that

\begin{eqnarray*}
d_D(w,q) \ge  \inf_{\gamma \in \Gamma(q,w)} \int_{\gamma^{-1}(G)} F_D(\gamma(t),d\gamma(t))dt \ge N \inf_{\gamma \in \Gamma(q,w)} \int_{\gamma^{-1}(G)} \vert d\gamma(t) \vert dt.
\end{eqnarray*}
Given a path $\gamma$, the last integral represents the Euclidean length  of the part of $\gamma([0,1])$ contained in $G$. 

\vspace{1mm}
\noindent{\bf Claim.} For $w \in G$ one has 
$$\inf_{\gamma \in \Gamma(q,w)} \int_{\gamma^{-1}(G)} \vert d\gamma(t) \vert dt \geq \vert x(w) - x(q) \vert.$$

\vspace{1mm}
To prove the Claim, there are two cases to consider.
First, if the path $\gamma$ is contained in $G$, then obviously its length is not smaller than $\vert x(w) - x(q) \vert$. 
Second, if $\gamma$ intersects the boundary of $G$, then the length of a connected component of $\gamma$ joining $q$ and a boundary point of $G$ is larger than or equal to $1$, which is larger than or equal to $\vert x(w) - x(q) \vert$. This proves the Claim.

\vspace{1mm}
Finally, if $w$ is not in $G$ (for example, if $w$ is not in $U$), then for every $\gamma \in \Gamma(q,w)$, the length $\int_{\gamma^{-1}(G)} \vert d\gamma(t) \vert dt$ is bounded from below by 1.
Therefore, we have the estimates

\begin{eqnarray}
\label{est2}
d_D(w,q) \ge N \min \{ 1, \vert x(w) - x(q) \vert \}, \,\,\, w \in D \cap U
\end{eqnarray}
and 
\begin{eqnarray}
\label{est3}
d_D(w,q) \ge N, \,\,\, w \in D \setminus U.
\end{eqnarray}
Now it follows from (\ref{est2}) and (\ref{est3}) that the condition  $w \in B_D(q,\delta)$  implies that $w \in U$ and that $\vert x(w) - x(q) \vert < \delta/N$.
\qed

\bigskip

{\bf Step 4. Precise estimate on $F_D$.}  
Consider a smooth function $\psi$ defined quite similarly as in Step 1 
and satisfying $\psi(t) = t$ for $t \le 1/2$ and $\psi(t) = 1$ when $t \ge 1$.  For every $w \in D \cap x^{-1}(\B^n)$ and every $A, \lambda > 0$ and $0 < \beta \leq \sqrt{B}$, consider the function 
$$
\Phi_{A,\lambda,\beta,w}(x) = \psi(\vert x - x(w) \vert^2/\beta^2) \exp(A\psi(\vert x - x(w) \vert)) \exp(\lambda \rho(x)).
$$
This function is well-defined on $D \cap U$ and takes its values in $[0,e^A]$.

- According to Step 1, we can choose $A$ and $C_0>0$ such that the function $\log \psi (|x-x(q)|^2/\beta^2) + A \psi(|x-x(q)| + C_0 |x|^2$ is MPSH on $\{|x-x(q)|^2/\beta^2 \leq 1/2\}$.

- On the set $\{|x-x(q)|^2/\beta^2 \geq 1\}$, we have $\log \Phi_{A,\lambda,\beta,w} (x) = A \psi(|x-x(q)|) + \lambda \rho$ and there exists $C_1>0$, depending only on $\psi$, such that $A \psi(|x-x(q)|) + C_1 |x|^2$ is MPSH on $\{|x-x(q)|^2/\beta^2 \geq 1\}$.

- On the set $\{1/2 \leq |x-x(q)|^2/\beta^2 \leq 1\}$, we have $1/2 \leq \psi(|x-x(q)|^2/\beta^2) \leq 1$ and there exists $C_2 > 0$, depending only on the first and second order derivatives of $\psi$, such that all the second order derivatives of the function $\log \psi(|x-x(q)|^2/\beta^2)$ are bounded between $-C_2/\beta^2$ and $C_2/\beta^2$. Hence there exists $C_3>0$ such that the function $\log \psi(|x-x(q)|^2/\beta^2) + A \psi(|x-x(q)|) + (C_3/\beta^2)|x|^2$ is MPSH on $\{1/2 \leq |x-x(q)|^2/\beta^2 \leq 1\}$.

Recall that we have supposed  $\beta^2 \leq B$; hence we can assume, increasing $C_3$ if necessary, that the function $\log \psi(|x-x(q)|^2/\beta^2) + A \psi(|x-x(q)|) + (C_3/\beta^2)|x|^2$ is MPSH on $D \cap U$. Hence the function $\log \Phi_{\lambda,\beta,w} + (C_3/\beta^2 - \lambda\varepsilon) \vert x \vert^2$ is MPSH in $D \cap U$. Now set 
$\lambda = C_3/(\varepsilon\vert \rho(w) \vert)$ and $\beta^2 = \vert \rho (w) \vert$ . Since $w \in D \cap x^{-1}(\mathbb B^n)$ we have $\beta^2 \leq B$ indeed, and we obtain a function denoted by $\Phi_w$ such that $\log \Phi_w$ is MPSH on $D \cap U$.

%There exists a constant $C > 0$ depending only on the function $\psi$  such that the function $\log \Phi_{\lambda,\beta,w} + (C/\beta^2 - \lambda\varepsilon) \vert x \vert^2$ is MPSH in $D \cap U$. Now set 
%$\lambda = 1/\vert \rho(w) \vert$ and $\beta^2 = C \vert \rho(w) \vert / \varepsilon$. We obtain a function denoted by $\Phi_w$, with $\log \Phi_w$ of class MPSH on $D \cap U$.

Set $r = (e^{2N} -1)/(e^{2N} + 1)$, so that the Poincar\'e radius of the disc $r\D$ in $\D$ is equal to $N$. It follows by Lemma \ref{LocLemma} that for each conformal harmonic mapping $g:\D \longrightarrow D$ such that $w = g(0)  \in D \cap x^{-1}(\B^n)$, one has the inclusion $g(r\D) \subset 
D \cap x^{-1}(2\B^n)$. Let $f: \D \longrightarrow D$ be a conformal harmonic mapping satisfying $f(0) = w$ and $df(0)e_1 = \alpha^{-1} \xi$, where $\alpha > 0$ and 
$\xi \in T_wD$. Then the function 
$$v(\zeta) = \Phi_w(f(\zeta))/\vert \zeta \vert^2$$
is subharmonic on $r\D \setminus \{ 0 \}$. Furthermore, $\lim\sup_{\zeta \to 0} v(\zeta) = \vert \xi \vert^2/(e^{C_3/\varepsilon} \vert \rho(w) \vert \alpha^2)$ (we choose normal coordinates at $w$ as above). 
Therefore, $v$ extends on $r\D$ as a subharmonic function. By the maximum principle, we have : $\alpha \ge e^{-A/2} e^{-C_3/2\varepsilon} r \vert \xi \vert \vert \rho(w) \vert^{-1/2}$.
By the definition of the Kobayashi-Royden metric we obtain the estimate
$$
F_D(w,\xi) \ge e^{-A/2} e^{-C_3/2\varepsilon} r \vert \xi \vert \vert \rho(w) \vert^{-1/2}.
$$
This completes the proof of Theorem~\ref{Kobest-thm}, setting $C=e^{-A/2} e^{-C_3/2\varepsilon} r$. \qed

\vspace{2mm}
As a first application, we obtain the following result.

\begin{thm}
\label{DiscTheo}
Let $(M,g)$ be a Riemannian manifold, $\rho$ be an MPSH function on $M$, and let $u: \D \to M$ be a 
conformal harmonic immersed disc such that $\rho \circ u \ge 0$ on $\D$ and $(\rho \circ u)(\zeta) \to 0$ as $\zeta \in \D$ tends to an open arc $\gamma \subset b\D$.
Assume that for a certain point $a \in \gamma$ the cluster set $C(f,a)$ contains a point $p \in M$ such that $\rho$ is strictly MPSH  near $p$.
Then $u$ extends to a neighborhood of $a$ in $\D \cup \gamma$ as a H\"older 1/2-continuous map.
\end{thm}
The proof follows from Theorem \ref{LocTheo} via the same argument as in \cite{Ch-Co-Su} so we skip it. In particular, we have the following

\begin{cor}
Let $\rho$ be a strictly MPSH function on a Riemannian manifold $(M,g)$. Set $M^+ := \{\rho > 0\}$ and $\Gamma = \{\rho = 0\}$. 
Assume that $u:\D \to M^+$ is a conformal harmonic immersed disc such that the cluster set $C(u,\gamma)$ is contained in $\Gamma$ for some open non-empty arc $\gamma \subset b\D$. Then $u$ extends on $\D \cup \gamma$ as a H\"older 1/2-continuous map.
\end{cor}

\section{Complete hyperbolicity of strictly pseudoconvex domains}

In this section $\nabla$ denotes the gradient,  We recall that :
\begin{itemize}
\item[(a)] a manifold $M$ is (Kobayashi) hyperbolic if the pseudistance $d_g$ is a distance on $TM$,
\item[(b)] a manifold $M$ is complete hyperbolic if the metric space $(M,d_{g})$ is complete.
\end{itemize}

%A $C^2$ function $\rho: M \to \R$ is called strictly MPSH if for every point $p \in M$ there exist local normal coordinates $x$ centered at $p$ and $\varepsilon > 0$ such that $\rho - \varepsilon \vert x \vert^2$ is an MPSH function in a neighborhood of $p = 0$. Clearly, such a positive constant $\varepsilon$ can be choosen independently of $p \in K$, for every compact subset $K \in \Omega$.

Here we obtain one  of our main results.

\begin{thm}
\label{TheoNorm}
Let $\Omega$ be a  relatively compact domain in $(M,g)$. Assume that $\rho$ is a  strictly MPSH $C^2$ function in a neighborhood of  $\overline\Omega$
such that $\Omega  = \{ \rho < 0 \}$ and 
$d\rho \neq 0$ on $b\Omega$. Then $\Omega$ is a complete hyperbolic domain.
\end{thm}
The proof is based on the approach of S. Ivashkovich - J.P. Rosay in \cite{Iv-Ro}. The following lemma is proved in \cite{Iv-Ro} for 
pseudoholomorphic discs. The proof is the same for conformal harmonic discs, so we drop it.

\begin{lemma}
\label{IvRoLem}
Let $\Omega$ be a domain in $(M,g)$. Let $p \in b\Omega$. Let $\phi$ be either:
\begin{itemize}
\item[(a)] a $C^1$ map from $\overline\Omega$ into $\R^2$ with $\phi(p) = 0$ and $\phi \neq 0$ on $\Omega$, or
\item[(b)] a $C^1$ map from a neighborhood $U$ of $p$ into $\R^2$, such that $\phi(p) = 0$ and 
$\phi \neq 0$ on $U \cap \overline\Omega \setminus \{ p \}$.
\end{itemize}
Let $\delta$ be a positive function defined on $(0,+\infty)$, and satisfying $\int_0^1 \frac{dt}{\delta(t)} = + \infty$. Assume that for
 every conformal harmonic map $u:\D \to \Omega$, such that $u(0)$ is close to $p$ one has 
 \begin{eqnarray*}
 \vert \nabla (\phi \circ u)(0) \vert \le \delta(\vert \phi \circ u)(0) \vert)
\end{eqnarray*}
 (here $\nabla$ denotes the gradient). Then $p$ is at infinite Kobayashi distance from the points in $\Omega$.
 \end{lemma}

We begin with the following localization principle.

\begin{lemma}
\label{LocLemma2}
Let $p_0$ be a boundary point of $\Omega$. For any neighborhood $V$ of $p_0$ there exists a neighborhood $W$ of $p_0$ and $r \in (0,1)$ with the following property:
if $u: \D \to \Omega$ is a conformal harmonic immersed disc that $u(0) \in W$, then $u(r\D) \subset V$.
\end{lemma}

\proof
Consider in the unit disc $\D$ the ball $B_P(0,t)$ with respect to the Poincar\'e metric, centered at the origin and of radius $t$. It follows from \cite{Ga-Su} that 
$u(B_P(0,t) \subset B_\Omega(u(0),t)$ (the Kobayashi ball in $\Omega$). Now the result of Lemma~\ref{LocLemma2} follows by Lemma \ref{LocLemma}. \qed

\vspace{2mm}
We present another localization principle for the Kobayashi - Royden metric.%, when it is not precised explicitely.

\begin{lemma}
\label{CompLem1}
Under the hypothesis of Theorem \ref{TheoNorm}, let $p_0$ be a boundary point of $\Omega$. Then there exist $0 < r < 1$, $\delta > 0$ and $C > 0$ 
with the following property: if a conformal harmonic immersed disc $u: \D \longrightarrow \Omega$ satisfies $dist(u(0),p_0) < \delta$, then
\begin{eqnarray}
\label{Comp1}
dist(u(0),u(\zeta)) \le C dist(u(0), b\Omega)^{1/2}
\end{eqnarray}
for $\vert \zeta \vert < r$.
\end{lemma}
\proof %Fix   $r < r_1 < 1$ small enough, such that $r_1$ satisfies the localization Lemma \ref{LocLemma2}.
Consider a coordinate neighborhood $U$ of $p_0$ with normal coordinates centered at $p_0 = 0$. We identify  $U$ with a ball in $\R^n$ and assume without loss of generality that the metric $g$ is a small deformation of the standard Euclidean metric $g_{st}$ on that ball. There exists a neighborhood $V \subset U$ of $p_0$ and a small enough 
 $\varepsilon > 0$   such that for each $p \in V$ the function $q \mapsto   \rho_p(q) = \rho(q) - \varepsilon \vert q - p \vert^2$, as well as the functions $q \mapsto \vert q - p \vert^2$ are MPSH on $V$. Here we use the notation $\vert \cdot \vert$ for the norm induced on $U$ by the local coordinates $x$ and the Euclidean metric in $\R^n$.
Also there exists $A,B > 0$ such that

 \begin{eqnarray*}
  -B \vert q - p \vert \le \rho_p(q) \le -A \vert q - p \vert^2.
\end{eqnarray*} 
According to Lemma~\ref{LocLemma2} there exist $\delta > 0$ and $0 < r < 1$ such that if $u: \D \to \Omega$ is a conformal harmonic immersed map such that  $dist(u(0),p_0) < \delta$, then
$u(\zeta) \in V$ when $\vert \zeta \vert \le r$. Choose $p \in b\Omega$ such that $dist(u(0), b\Omega) = dist (u(0),p)$. Since the function $\vert u(\zeta) - p\vert^2$ is subharmonic on $\{\zeta \in \C : / \vert \zeta \vert \le r\}$, by the mean inequality we have
 
 \begin{eqnarray*}
 \vert u(\zeta) - p \vert^2 \le \frac{1}{2\pi} \int_0^{2\pi} \vert u(r e^{i\theta}) - p \vert^2d\theta.
 \end{eqnarray*}
 Again by the mean value property we have
 
 \begin{eqnarray*}
 -(\rho_p \circ u)(0) \ge \frac{1}{2\pi} \int_0^{2\pi} - (\rho_p \circ u)(r e^{i\theta})d\theta  \ge \frac{A}{2\pi} \int_0^{2\pi}\vert u(r e^{i\theta}) - p \vert^2d\theta  \ge A  \vert u(\zeta) - p \vert^2.
 \end{eqnarray*}
 Therefore
 $$dist(u( \zeta),u(0))  \le dist(u(0),p) + dist(p, u( \zeta)  \le C dist(u(0),p)^{1/2}$$
 for some constant $C > 0$ which proves Lemma~\ref{CompLem1}. \qed
 
 \vspace{2mm}
 We also have the following version of the Schwarz lemma for harmonic maps (a more general result is contained in  \cite{Jo1}, Th. 4.8):
 
  \begin{lemma}
   \label{Schwarz}
 Let $ \Omega$ be a bounded subset in $ (\R^n,g)$ and let $p \in b\Omega$. There exist $0 < r' < r < 1$, $C > 0$ and a neighborhood $V$ of $p$ such that if $u:  \D  \to  \Omega$ is a harmonic disc 
  (with respect to the metric $g$) and $u(r\D)  \subset V$, then
   \begin{eqnarray}
    \label{Comp2}
     \vert  \nabla u(\zeta)  \vert  \le C  \sup_{ \vert  \omega  \vert < r}  \vert u( \omega) - u(0)  \vert
      \end{eqnarray}
      if $ \vert  \zeta  \vert  \le r'$.
       \end{lemma}

\vspace{2mm}
\noindent{\sl Proof of Theorem~\ref{TheoNorm}.} As above, we consider local normal coordinates centered at $p = 0$, where the neighborhood $U$ is chosen sufficiently small so that the assumptions of Lemma~\ref{Schwarz} are satisfied on $U$. 
 Using dilations, we may assume that $U$ contains the unit ball $ \B^n$ and that $g$ is close enough to 
 $g_{st}$ in the $C^k$ norm  on $U$, with $k$ large enough. Furthermore, we may assume that $U  \cap  \overline{ \Omega}  \setminus  \{ 0  \}$
 is contained in $ \{ x  \in  U: x_1 < 0  \}$.
 
Let $u=(u^1,\dots,u^n):  \D  \to  \Omega$ be a conformal harmonic immersion. We assume that $u(0)$ is close enough to  $p$, so by the localization principle Lemma~\ref{LocLemma2}, there exists $0 < r < 1$ such that 
$u(r\D)  \subset U$. Shrinking $r$ if necessary, we obtain by Lemmas \ref{Schwarz} and \ref{CompLem1}
 \begin{eqnarray}
 \label{Comp4}
  \vert  \nabla  u( \zeta) \vert  \le C dist(u(0), b \Omega)^{1/2}  \le C (-u^1(0))^{1/2}
   \end{eqnarray}
  if $ \vert  \zeta  \vert  \le r'$.
  
  Rescaling the disc $\D$, we can assume that the estimate~ (\ref{Comp4}) holds on $\D$.
  
  Our goal is to prove that changing $C$ if necessary, we have
  
  \begin{eqnarray}
  \label{Comp3}
  \vert \nabla u^1(0) \vert \le C \vert u^1(0) \vert.
  \end{eqnarray}
  
Apply the formula (\ref{eq2}) to the function $\rho = x_1$. Then for every $\zeta \in \D$

\begin{eqnarray*}
 \Delta u^1(\zeta)  =  
    - \sum_{j,k}\Gamma^1_{j,k}(u(\zeta)) \frac{\partial u^j}{\partial \zeta} (\zeta) \frac{\partial u^k}{\partial \overline \zeta}(\zeta).
\end{eqnarray*}

Using the estimate (\ref{Comp4}) we conclude that 

\begin{eqnarray}
\label{Comp5}
\vert \Delta u^1(\zeta) \vert \le A \vert u^1(0) \vert
\end{eqnarray}
where $A > 0$ is a constant.
Consider the function $h$ harmonic on $\D$:

\begin{eqnarray}
\label{Comp6}
h(\zeta) = u^1(\zeta) - \frac{1}{2\pi}\int_{\R^2}\Delta  u^1(\omega ) \ln \vert \omega - \zeta \vert dm(\omega) - A \vert u^1(0) \vert.
\end{eqnarray}
Here we suppose that $\Delta u^1$ is extended by $0$ outside $\D$. Using (\ref{Comp5}) we see that 
$$
\left\vert \frac{1}{2\pi}\int_{\R^2} \Delta  u^1(\omega+\zeta) \ln \vert \omega \vert dm(\omega)\right\vert \le A \vert u^1(0) \vert.
$$
Since $u^1 < 0$, we obtain that $h < 0$. Furthermore, we have $\vert h(0) \vert < (2A + 1) \vert u^1(0) \vert$. 

Now it follows from the classical Schwarz lemma for negative harmonic functions that $\vert \nabla h(0) \vert \le 2 \vert h(0) \vert$.
Therefore from (\ref{Comp6}) we obtain 
\begin{eqnarray*}
\vert \nabla u^1(0) \vert \le \vert \nabla h(0) \vert + C \sup \vert \Delta u^1 \vert \le \left(2(2A+1) + CA\right) \vert u^1(0) \vert .
\end{eqnarray*}
This proves (\ref{Comp3}). Now Theorem~\ref{TheoNorm} follows exactly as in the proof of Theorem 1 in \cite{Iv-Ro}, using  Lemma \ref{IvRoLem}. \qed

\vspace{2mm}
As a direct application of Theorem~\ref{TheoNorm}, we have the following

%\begin{cor}\label{complete-cor}
%Let $\Omega=\{\rho < 0\}$ be a relatively compact domain in a Riemannian manifold $(M,g)$. We assume that $\rho$ is of class $C^2$ in a neighborhood of $\overline{\Omega}$. If there exists $c > 0$ such that $\nabla^2\rho(p)(v,v) \geq c g(v,v)$ for every $p \in \Omega$ and every $v \in T_pM$, then $\Omega$ is complete hyperbolic.
%\end{cor}

%As an application of Corollary~\ref{complete-cor} we have the following examples of complete hyperbolic domains in Riemannian manifolds.

\begin{cor}
%\begin{itemize}
%\item[(i)] Let $p \in \mathbb R^n$ and let $r>0$. Then small $C^2$ deformations of the Euclidean ball $B_{Eucl}(p,r)$ are complete hyperbolic.
Let $(M,g)$ be a Riemannian manifold. For every $p \in M$, the ball $B_g(p,r)$ is complete hyperbolic for every sufficiently small $r>0$.
%\item[(iii)] Let $(M,g)$ be a Riemannian manifold with non positive Riemannian sectional curvature. Then for every $p \in M$ and every $r > 0$, the ball $B_g(p,r)$ is complete hyperbolic.
%\end{itemize}
\end{cor}

%\begin{proof}
%Point (i). For every $p \in \mathbb R^n$, the function $f_p : x \mapsto \vert x-p \vert^2$ satisfies $\nabla^2f_p(x)(v,v) \geq \|v\|^2$ for every $x,v \in \R^n$. Let $r'>r$ and let $\rho$ be a $C^2$ function defined in a neighborhood of $\overline{B_{Eucl}(p,r')}$ and such that $\|\rho - f_p\|_{C^2\left(\overline{B_{Eucl}(p,r')}\right)}$ is sufficiently small. Then there exists $c > 0$ such that for every $x \in B_{Eucl}(p,r')$ and every $v \in \R^n$, $\nabla^2 \rho(x)(v,v) \geq c\|v\|^2$.

%Point (ii). It follows from Example~\ref{mpsh-ex} that there exists a neighborhood $U$ of $p$ and $c>0$ such that the function $f_p : x \in M \mapsto d^2_g(x,p)$ satisfies, for every $x \in U$ and every $v \in T_xM$ : $\nabla^2f_p(x)(v,v) \geq c g(v,v)$.

% Point (iii). According to Theorem 4.1 (2) in \cite{Bi-ON}, the funtion $f_p : x \in M \mapsto d^2_g(x,p)$ is strictly convex in a Riemannian manifold with non positive Riemannian sectional curvature.
% \end{proof}

\section{Theorem of Fatou}

As an application of the localization principle and of the estimates of the Kobayashi metric proved in Section 3, we obtain the existence of tangential limits almost everywhere for bounded conformal harmonic immersed maps from the unit disc to a Riemannian manifold. First we present the following

\begin{definition}
A relatively compact domain $\Omega$ in a Riemannian manifold $(M,g)$ is called bounded if in a neighborhood of $\overline\Omega$ there exists a bounded strictly MPSH function $\rho$ of class $C^2$. A map $u: \D \to M$ is called bounded if the closure $\overline{u(\D)}$ of its image is contained in a bounded domain.
\end{definition}

The main result of this section is the following Fatou type Theorem :

\begin{thm}\label{Fatou-thm}
Let $u$ be a bounded conformal harmonic immersed map from the unit disc $\D$ to a Riemannian manifold $(M,g)$. Then, for almost every $e^{i\theta} \in b\D$, $u$ admits a non tangential limit at
$e^{i\theta}$.
%Then for almost every $e^{i\theta} \in b\D$ the limit $u^*(e^{i\theta})$ of $u(\zeta)$ exists when $\zeta \in \D$ approaches $e^{i\theta}$ non-tangentially.
\end{thm}

Let us precise what we mean by non-tangential approach. Given $\theta \in [0,2\pi]$ and $0 < \alpha < 1$, consider the domain 
$$K(\theta,\alpha) = \{ \zeta \in \D: \vert \zeta - e^{i\theta}\vert \zeta \vert \vert < \alpha(1 - \vert \zeta \vert) \}.$$
Geometrically this is a cone around the radius $[0,e^{i\theta}]$ with vertex at $e^{i\theta}$. 

\begin{definition}
We say that $u$ admits a non-tangential limit at $e^{i\theta}$ if for every $\alpha \in (0,1)$ the map $u(\zeta)$ admits  a limit $u^*(e^{i\theta})$ 
as $\zeta \to e^{i\theta}$, $\zeta \in K(\theta,\alpha)$. 
\end{definition}

There are many various generalizations of this result and we do not try to discuss them. Note that recently another version of the Fatou theorem for bounded harmonic maps in Riemannian manifolds is obtained in \cite{Be-Hu}, using a quite different approach.

\vspace{2mm}
\noindent{\sl Proof of Theorem~\ref{Fatou-thm}.} Our first remark is the following.
Since the map $u$ is bounded by assumption, there exists a relatively compact domain $\Omega$ in $M$ and a bounded strictly MPSH $C^2$-function $\rho$ in $\Omega$, such that $\overline{u(\D)}$  is contained in $\Omega$.
%We need the localization principle for the Kobayashi-Royden metric (Theorem 3.1 of \cite{Ga-Su2}):

%{\it Let $x: U \to 3\B^n$ be a normal coordinate neighborhood in $M$ centered at a point $p \in D$ (in particular, $x(p) = 0$). Assume that $U$ is small enough such that $\vert x \vert^2$ is an MPSH function on $U$. Let $\rho$ be a negative MPSH function  in $D$ such that the function $\rho - \varepsilon \vert x \vert^2$ is MPSH on $D \cap U$ and $\vert \rho \vert \le B$ in $D \cap x^{-1}(2\B^n)$ for some constants $\varepsilon, B > 0$.
%Then there exists a constant $C=C(\varepsilon,B) > 0$, independent of $\rho$, such that for every $w \in D \cap x^{-1}(\B^n)$ and every tangent vector $\xi \in T_w(M)$:
%\begin{eqnarray*}
%F_D(w,\xi) \ge C \vert \xi \vert \vert \rho(w) \vert^{-1/2}.
%\end{eqnarray*}}

%Here we use the notation $\vert \cdot \vert$ for the norm induced on $U$ by the local coordinates $x$ and the Euclidean metric. The coordinate neighborhood $U$ is not assumed to be contained in $D$. Therefore, this result gives a first (non optimal) asymptotic behavior estimate of $F_D$ near the boundary of $D$ in $M$. Note also that no conditions such as boundedness or Kobayashi hyperbolicity are imposed on $D$.

Denote by $P_{\D}$ the Poincar\'e metric, defined for $\zeta \in \D$ and $ \lambda \in \C$ by

$$
P_{\D}(\zeta,\lambda) = \frac{\vert \lambda \vert}{1- \vert \zeta \vert^2}.
$$
It is proved in \cite{Ga-Su} that  for every conformal harmonic immersed disc $u:\D \to M$ we have, for every $\zeta \in \D$ and every $\lambda \in \C$:

\begin{eqnarray}\label{invariance}
F_M(u(\zeta),du(\zeta)\cdot \lambda) \le P_{\D}(\zeta,\lambda).
\end{eqnarray}

Let $K$ be a compact subset of $\Omega$. Using Theorem~\ref{Kobest-thm} and the decreasing property (\ref{invariance}), we obtain for any conformal harmonic immersed disc $v:\D \to K$ and any $\zeta \in \D, \ \lambda \in \C$ :

$$
C \frac{\|dv(\zeta)\cdot \lambda\|_g}{|\rho(v(\zeta))|^{1/2}} \leq F_{\Omega}(v(\zeta),dv(\zeta)\cdot \lambda) \leq \frac{|\lambda|}{1-|\zeta|^2}.
$$

We conclude that for any compact subset $K$ of $\Omega$ 
there exists a constant $C'= C'(K,g,\rho)$ such that for any conformal harmonic immersed disc $v:\D \to K$, one has the estimate

%\begin{eqnarray}
%\label{FatouSchwarz1}
%\parallel dv(0) \parallel_g \le C'.
%\end{eqnarray}
%Applying a conformal automorphism of $\D$, we obtain the estimate

\begin{eqnarray}
\label{FatouSchwarz2}
\parallel dv(\zeta) \cdot \lambda \parallel_g \le C' \frac{|\lambda|}{1- \vert \zeta \vert^2}
\end{eqnarray}
for any $\zeta \in \D$, $\lambda \in \C$.

Integrating along piecewise $C^1$ smooth paths, we get  for all $\zeta_1, \zeta_2 \in \D$ the following estimate

\begin{eqnarray}
\label{FatouSchwarz3}
dist_g (v(\zeta_1),v(\zeta_2)) \le C' \rho_{\D}(\zeta_1,\zeta_2)
\end{eqnarray}
where $\rho_{\D}$ denotes the Poincar\'e distance on $\D$.

The second tool of the proof is the classical Littlewood theorem (see \cite{Li}). This theorem asserts that a subharmonic function, bounded from above on the unit disc, admits radial limit values a.e. on $b\D$. It is known that without additional assumptions in general there are no non-tangential limits a.e. A sufficient condition for the existence of radial limits a.e. is the condition of normality of a subharmonic function, see \cite{Me-Ja}.  Let $f$ be a real or complex valued function on $\D$. Consider the family 
$$\Phi = \{ f \circ \phi: \phi \in Aut(\D)\}$$
where $Aut(\D)$ denotes the group of conformal automorphisms of $\D$.  A function $f$ is called normal if the family $\Phi$ is normal i.e. every sequence of functions from $\Phi$ contains a subsequence which either converges uniformly on compact subset of $\D$, or else diverges uniformly to infinity on every compact subset of $\D$.

Since $\rho$ is of class $C^2$, the function $\rho \circ u$ is normal in view of (\ref{FatouSchwarz3}). Hence, this function admits non-tangential boundary values a.e. 

Now let $\phi = (\phi_1,...,\phi_N): V \to \R^N$ be a smooth injective map from a neighborhood $V$ of $\overline\Omega$ to $\R^N$ for some integer $N$ large enough. Since the function 
$\rho$ is a $C^2$ strictly    MPSH function in a neighborhood of $\overline\Omega$,   for $\varepsilon > 0$ small enough the functions $\rho + \varepsilon \phi_j$ are also MPSH functions for $j=1,...,N$. We apply the above normality  argument to the subharmonic functions $w_\varepsilon = \rho \circ u + \varepsilon \phi_j \circ u$, $j= 1,...,N$. We conclude that every function $\phi_j \circ u$ has non-tangential boundary values a.e. By injectivity, $u$ admits non-tangential boundary values $u^*(e^{i\theta})$ a.e. This proves Theorem~\ref{Fatou-thm}. \qed

\section{Picard theorems and hyperbolicity}

In this section we establish an analog of the Picard theorem for conformal harmonic maps and give some examples of hyperbolic manifolds.

\subsection{Degeneracy locus of the Kobayashi pseudodistance} Let $(N,g)$ be a Riemannian manifold and $(M,g)$ be a relatively compact submanifold of $N$. We extend the Kobayashi pseudodistance $d_M$ on the closure $\overline M$ of $M$ in $N$ as follows: for $p,q \in \overline M$ we set

\begin{eqnarray*}
d_{\overline M}(p,q) = \lim \inf_{\tilde p \to p, \tilde q \to q} d_M(\tilde p,\tilde q), \, \, \tilde p,\tilde q \in M.
\end{eqnarray*}
We say that $p \in \overline M$ is a degeneracy point for the Kobayashi pseudodistance $d_{\overline M}$ if there exists a point $q \in \overline M \setminus \{ p \}$ such that $d_{\overline M}(p,q) = 0$. We denote by $\Sigma_M(N)$ the set of degeneracy points of $d_{\overline M}$. A point $p \in \overline M$ is called hyperbolic for $M$ if there exists 
a neighborhood $U$ of $p$ in $N$ and a  constant $C > 0$ such that $F_M \ge C \sqrt{g}$ on $U \cap M$.

\begin{prop}
\label{HypDeg}
For a point $p \in \overline M$ the following conditions are equivalent:
\begin{itemize}
\item[(i)] $p$ is not in $ \Sigma_M(N)$;
\item[(ii)] $M$ is hyperbolic at $p$. %$p$ is a hyperbolic point for $M$.
\end{itemize}
\end{prop}

For the proof we need the following Schwarz lemma which easily follows from standard elliptic estimates for second order quasilinear operators,  see \cite{Jo1}, Theorem 4.8.1 p. 113 :

\vspace{1mm}
\noindent{\bf Theorem (Schwarz Lemma).}
{\sl Let $(X,h)$ and $(Y,g)$ be Riemannian manifolds. Let $\B(x_0,R_0) \subset X$, $R_0 < \min\left(i_X(x_0),\frac{\pi}{2\kappa_X}\right)$, where $i_X(x_0)$ denotes the injectivity radius at $x_0$ and $-\omega_X^2 \leq K_X \leq \kappa_X^2$ are curvature bounds on $\B(x_0,R_0)$. Let $\B(y_0,R') \subset Y$, $R' < \min\left(i_Y(y_0),\frac{\pi}{2\kappa_Y}\right)$, where $i_Y(y_0)$ denotes the injectivity radius at $y_0$ and $-\omega_Y^2 \leq K_Y \leq \kappa_Y^2$ are curvature bounds on $\B(y_0,R')$.

\vspace{1mm}
If $u : X \rightarrow \B(y_0,R')$ is harmonic, then for all $R \leq R_0$:
\begin{equation}\label{harm-ineq}
|\nabla u(x_0)| \leq c_0 \max_{x \in B(x_0,R)}\frac{dist_g(u(x),u(x_0))}{R}
\end{equation}
where $c_0 = c_0(R_0,\omega_X, \kappa_X, \dim X, R', \omega_Y, \kappa_Y, \dim Y)$.
}
Here $\nabla$ denotes the gradient.

%\begin{lemma}
%\label{Elliptic}
% Let $\Omega$ be a domain in $\R^n$. There exists $\delta_0 > 0$ such that for every $0 < r < 1$ and every Riemannian metric $g$ satisfying 
 %$\| g - g_{st}\|_{C^2(\Omega)} \le \delta_0$, there exists $c = c(\delta_0,r) > 0$ such that 
 %\begin{eqnarray*}
%\| u\|_{C^1(r\D)} \le c\| u\|_{C^0(r\D)}
% \end{eqnarray*}
% for each conformal harmonic map $u: \D \to \Omega$.
 %\end{lemma}
 
 Proof of Proposition~\ref{HypDeg}. We prove that (i) implies (ii). Arguing by contradiction, assume that  $p$ is not a hyperbolic point for $M$. Then there exists a sequence $(p_k, \xi_k)$ in 
 the tangent bundle $TM$, such that $p_k \to p$, $\vert \xi_k \vert_g = 1$ and $F_M(p_k,\xi_k) \to 0$. Then there exists a sequence of conformal harmonic maps 
 $f_k: \D \to M$ such that $f_k(0) = p_k \to p$ and $\vert df_k(0) \cdot e_1\vert_g \to \infty$. Consider a sufficiently small open relatively compact coordinate neighborhood $U$ around 
 the point $p$. Assume that for some $r > 0$ one has the inclusion $f_k(r\D) \subset U$. Then by the above Schwarz lemma   the sequence $(\vert df_k(0) \cdot e_1\vert_g)$ is bounded, which is a contradiction. Hence, passing to a subsequence, we can assume that there exists a sequence $z_k \to 0$ in $\D$ such that $f_k(z_k) \to q \in \partial U$. Therefore
 
\begin{eqnarray*}
d_{\overline M} (p,q) \le \lim_{k \rightarrow \infty} d_M(f_k(0),f_M(z_k)) \le \lim_{k \rightarrow \infty} \rho_{\D}(0,z_k) = 0
\end{eqnarray*}
which contradicts the non-degeneracy of $d_{\overline M}$ at $p$.

Now we prove that (ii) implies (i). Again arguing by contradiction, assume that $p$ is a degeneracy point.  Hence there exists a point $q \in \overline{M}$, $q \neq p$, such that 
$d_{\overline M}(p,q) = 0$. Since (ii) holds, there exists a neighborhood $U$ of $p$ in $M$ and a constant $C > 0$ such that $q$ is not in $\overline U$ and $F_M(x,\xi) \ge c |\xi|_g$ for each $x \in M \cap U$ and $\xi \in T_xM$. Consider neighborhoods $U_1$ and $U_2$ of $p$ and $q$ respectively such that $\overline{U_1} \subset U$ and 
$\overline U \cap \overline U_2$ is empty. Let $\tilde p \in U_1$ and $\tilde q \in U_2 \cap M$ be arbitrary points. Let $\gamma: [0,1] \to M$ be a picewise smooth path with $\gamma(0) = \tilde p$ and $\gamma(1) = \tilde q$. Then 

\begin{eqnarray*}
d_M(\tilde{p},\tilde{q}) = \inf_{\gamma} \int_0^1 F_M(\gamma(t), \dot{\gamma}(t)) dt \ge c \int_{\gamma^{-1}(U)}|\dot{\gamma} (t)|_g dt \ge c \, dist(\partial U_1, \partial U) = c_1 > 0
\end{eqnarray*}
 Therefore $d_{\overline M}(p,q) \ge c_1 > 0$, which is a contradiction. This proves Proposition~\ref{HypDeg}. \qed

  \begin{cor}
 The set  $ \Sigma_M(N)$ is closed.
  \end{cor}

 \subsection{The Picard theorem}
 Here we prove the main result of this section. Denote by $ \D^* =  \D  \setminus  \{ 0  \}$ the punctured unit disc.
 
  \begin{thm}
   \label{PicardThm}
   Let $(M,g)$ be a relatively compact submanifold of a Riemannian manifold $(N,g)$, and let $u_k:  \D^*  \to M$ be a sequence of conformal harmonic immersions. Suppose that $(w_k)$ is a sequence in $ \D^*$ converging to $0$ and such that the sequence $(u_k(w_k))$ converges to $q  \notin  \Sigma_M(N)$. Then  for every sequence $(z_k)$ in $\D^*$ converging to $0$, the sequence $(u_k(z_k))$ also converges to $q$.
    \end{thm}
   
Our approach is adapted from the one followed by F. Haggui and A. Khalfallah \cite{Ha-Kh, Ha-Kh2} who considered the case of almost complex manifolds. For the proof we need the following monotonicity property for minimal surfaces, see for example \cite{Ma}: 

 \begin{lemma}
  \label{MonotLem}
  Let $(M,g)$ be a compact Riemannian manifold. There exist positive constant $ \varepsilon$ and $c$ such that for every minimal surface $S$ in $M$ and every $0 <  \varepsilon  \le  \varepsilon_0$ one has
   \begin{eqnarray*}
   Area(S  \cap B(x, \varepsilon) ) \ge c  \varepsilon^2
    \end{eqnarray*}
    whenever $x  \in S$ and $S  \cap B(x,  \varepsilon)$ is a surface with boundary contained in the boundary of the ball $B(x,  \varepsilon)$.
     \end{lemma}
     
\noindent{\bf Proof of Theorem~\ref{PicardThm}.} Arguing by contradiction, assume that there exists a sequence $(z_k)$ converging to $0$ such that $u_k(z_k)$ converges to $ \tilde q  \neq q$.

{ \it Case 1.}  By taking subsequence and relabeling, assume that $ \vert w_k  \vert <  \vert z_k  \vert$ for every $k$. Consider the circles $\mathcal C_k=\{z \in \mathbb C : |z| = |w_k|\}$.% defined by $ \rho_k(t) = w_ke^{it}$ with  $t  \in [0,2 \pi]$.

{ \bf Claim} $u_k( \mathcal C_k)  \to q$.
Indeed, using the decreasing property of the Kobayashi distance with respect to conformal harmonic maps, we have for every $p_k  \in  \mathcal C_k$:
 \begin{eqnarray*}
   d_M(u_k(p_k),u_k(w_k))  \le  d_{ \D^*}(p_k,w_k)  \to_{k \to \infty} 0
    \end{eqnarray*}
    because $d_{ \D^*}(p_k,w_k) = O(-1/ \log \vert w_k  \vert)$. Since $q  \notin  \Sigma_M(N)$, we conclude that $u_k(p_k)  \to q$, which proves the claim.
    
  Choose relatively compact open coordinate neighborhoods   $U$ and $W$ of $q$ such that  $ \overline U  \subset W$, $U$ is diffeomorphic to the unit ball $B(0,1)$ in $ \R^n$, 
  $$W  \cap  \Sigma_M(N) =  \emptyset$$
and 
$$ \tilde q  \notin W.$$
Since $M$ is hyperbolic at $q$, there exists a constant $c > 0$ such that 
 \begin{eqnarray}
  \label{Picard1}
  F_M  \ge c \sqrt{g}
   \end{eqnarray}
on $W  \cap M$. 

Since $u_k( \mathcal C_k)  \to q$ and $u_k(z_k)  \to  \tilde q$, for sufficiently large $k$ we have $u_k( \mathcal C_k)  \subset U$ and $u_k(z_k)  \notin W$. %Hence, there exists $ \tilde z_k$ such that 
%$ \vert w_k  \vert <  \vert  \tilde z_k  \vert  <  \vert z_k  \vert$ and $u_k( \tilde z_k)  \in  \partial U$. Passing to a subsequence and dropping the tilda, we may assume that $u_k(z_k)  \to p  \in  \partial U$. We stress that $p  \notin  \Sigma_M(N)$.

Consider the largest open annulus $A_k$ containing the circle $ \mathcal C_k$ and such that $u_k(A_k)  \subset U$. Since $u_k(z_k)  \notin W$, there exist $a_k  \ge 0$ and $b_k <  \vert z_k  \vert$ such that 
$$A_k =  \{ z \in \mathbb C : a_k <  \vert z  \vert < b_k  \}.$$
Furthermore, there exists $\tilde z_k$ with $\vert \tilde z_k \vert = b_k$ such that $u_k(\tilde z_k) \in \partial U$. Passing to a subsequence and dropping tilde we may assume that $u_k(z_k) \to p \in \partial U$.
Now consider the annulus
$$ \tilde A_k =  \{ z \in \mathbb C :  \vert w_k  \vert <  \vert z  \vert <  \vert z_k  \vert  \}$$
and the circles 
$$ \mathcal C'_k =  \{ z \in \mathbb C :  \vert z  \vert =  \vert z_k  \vert  \}.$$
Recall that $u_k( \mathcal C_k)  \to q$ and similarly to the above claim, we also have $u_k( \mathcal C'_k)  \to p  \in  \partial U$. Hence for sufficiently large $k$ we have
$$u_k( \mathcal C_k)  \subset B(q,1/4)$$
and
$$u_k( \mathcal C'_k)  \subset W \setminus \overline B(q,3/4).$$
Therefore, there exist points $s_k \in \tilde A_k$ such that 
$$u_k(s_k) \in \partial B(q,1/2).$$
By the monotonicity lemma~\ref{MonotLem} (on the compact Riemannian manifold $\overline{M}$), there exist positive constants $\varepsilon_0$ and $c_1$ such that for $0 < \varepsilon < \inf (\varepsilon_0, 1/4)$ one has 
\begin{eqnarray*}
Area(u_k(\tilde A_k)) \ge Area(u_k(\tilde A_k) \cap B(u_k(s_k),\varepsilon)) \ge c_1 \varepsilon^2.
\end{eqnarray*}

 On the other hand, since the Kobayashi-Royden metric is not increasing with respect coformal harmonic immersions, we have in view  of (\ref{Picard1})
 
 \begin{eqnarray}
 \label{dec}
 \parallel du(\zeta) \xi \parallel_g \le (1/c) F_{\D^*}(\zeta, \xi)
 \end{eqnarray}
 where
$$F_{\D^*}(\zeta, \xi) = \frac{\vert \xi \vert}{\vert \zeta \vert \vert \log \vert \zeta \vert^2\vert}$$

 Since each map $u_k$ is conformal harmonic, the energy $E(u_k)$ over $\tilde A_k$ coincides with the area $Area(u_k(\tilde A_k))$ (see for example \cite{EM}). Hence from (\ref{dec}) we have 

\begin{eqnarray}
\label{Picard5}
Area(u_k(\tilde A_k)) \le c_1 \int_{\tilde A_k}  \frac{d\zeta \wedge d\overline\zeta}{\vert \zeta \vert^2 \vert \log^2 \vert \zeta \vert \vert} \le c_2 \vert 1/ \log \vert w_k \vert - 1/ \log \vert z_k \vert \vert  \to 0.
\end{eqnarray}
where $c_1,c_2 > 0$ are constants. This is a contradiction.

{\it Case 2.} By taking subsequence and relabeling, assume that $ \vert z_k  \vert <  \vert w_k  \vert$ for every $k$. Similarly to Case 1, there exists a sequence $(\tilde z_k)$ such that for every $k$
$\vert z_k \vert < \vert \tilde z_k \vert < \vert z_k \vert$, the annulus $\tilde A_k = \{ z: \vert \tilde z_k \vert < \vert z \vert < \vert w_k \vert \}$ satisfies $u_k(\tilde A_k) \subset U$ and $u_k(\tilde z_k) \to p \in \partial U$. Then the argument of the Case 1 goes through. This proves Theorem~\ref{PicardThm}. \qed

 As a consequence we obtain the following
 
 \begin{thm}
 \label{PicardThm2}
 Let $(M,g)$ be a relatively compact submanifold of a Riemannian manifold $(N,g)$, and let $u:\D^* \to M$ be a conformal harmonic immersion. Suppose that there exists a sequence $(z_k)$ of points in  $\D^*$ converging to $0$ such that $u(z_k) \to q \notin \Sigma_M(N)$. Then $u$ extends on $\D$ as a conformal harmonic map.
 \end{thm}
 Indeed, by Theorem~\ref{PicardThm}, $u$ extends continuously to the origin. Then $u$ is harmonic on the whole $\D$ by the elliptic regularity, see for example \cite{He}.
 
 We refer the reader to  \cite{Ko} for a discussion on the relation between the versions of Theorem \ref{PicardThm} for complex manifolds and the classical Picard theorem.

\subsection{Some examples of hyperbolic manifolds} In view of Theorem \ref{PicardThm} it seems natural to explore further the classes of Kobayashi hyperbolic Riemannian manifolds
and to give more examples. We begin with an analog of the Brody lemma.

The following result is due to M.Gromov (see \cite{Gr2} Section 1.D)
 
 \begin{prop}
\label{BlGrPro}
Suppose that $(M,g)$ is a compact Riemannian manifold such that every harmonic map $\R^2 \to M$ is constant. Then the space of 
harmonic maps $f:\D \to M$ is compact in $C^r$-topology for each $r = 0,1,2,...$. That is for every point $\zeta \in \D$ there exists a constant 
$C = C(\zeta, M, r)$ such that the $r$-th order differential (jet) of every harmonic map $f: \D \to M$ satisfies 
\begin{eqnarray}
\label{BlGr1}
\vert D^rf(\zeta)\vert_g \le C.
\end{eqnarray}
\end{prop}

Therefore we obtain the following (partial) analogue of Brody's theorem
\begin{thm}
\label{Brody}
Let $(M,g)$ be a compact Riemannian manifold. Assume that every harmonic map  $\R^2 \to M$ is constant. Then $M$ is Kobayashi hyperbolic.
\end{thm}
\begin{proof} Recall that $M$ is Kobayashi hyperbolic if and only if it is locally Kobayashi hyperbolic at every point of $M$. Assume by contradiction that there exists a point $p \in M$ such that $M$ is not locally hyperbolic near $p$. Then there exists a sequence of conformal harmonic maps $f_k: \D \to M$ such that $f_k(0) = p$ and $\vert df_k(0) \cdot e_1\vert_g \to \infty$ as $k \to \infty$. But this contradicts Proposition \ref{BlGrPro}.
\end{proof}

\bigskip
\bigskip

As usual, given two manifolds $M, N$, given a Riemannian metric $g$ on $M$ and a smooth map $f : N \rightarrow M$, we denote by $f^*g$ the $(0,2)$-tensor, pull-back of $g$ by $f$ and defined for every $p \in N$ and every $v,w \in T_pN$ by $f^*g(p)(v,w) = g(f(p))(df(p)\cdot v,df(p) \cdot w)$.
%The first examples we consider are domains in $\mathbb R^n$ endowed with classical complete Riemannian metrics or domains in $\mathbb C^n = \mathbb R^{2n}$ endowed with classical complete K\"ahler metrics.
If $g$ is a Riemannian metric (resp. K\"ahler metric), we denote by $K(g)$ (resp. $H(g)$) the sectional (resp. holomorphic sectional) curvature of $g$. In order to avoid any confusion and since we will consider different Riemannian metrics on the same manifold, we will denote by $F_{(M,g)}$ the Kobayashi metric on $M$ with respect to the Riemannian metric $g$. The computation of the  Kobayashi metric $F_{(M,g)}$ is based on the following version of the Schwarz Lemma for quasiconformal harmonic maps (see \cite{Go-Is} Theorem 2, p.565) :

{\bf Theorem \cite{Go-Is}.} {\sl Let $\D$ be endowed with the Poincar\'e
metric $g_{\D}$ of constant curvature $-4$. Let $M$ be a Riemannian manifold
with sectional curvature $K(g)$ bounded from above by a negative constant $-c$. If $u : \D \rightarrow M$ is a conformal harmonic mapping, then :
\begin{equation}\label{schwarz-eq}
u^*g \leq \frac{8}{c} g_{\D}.
\end{equation}
}
Here, for every $\zeta \in \D$, $g_{\D}(\zeta) = \frac{d\zeta \otimes d\overline{\zeta}}{(1-|\zeta|^2)^2}$.

Notice that the Poincar\'e metric $g_{\D}$ on $\mathbb D$ is conformal to the standard Euclidian metric $g_{st}$. Moreover, the Euler-Lagrange equations for harmonic maps are conformally invariant in real dimension two. Hence $u:(\D,g_{st}) \rightarrow (M,g)$ is conformal harmonic if and only if $u:(\D,g_{\D}) \rightarrow (M,g)$ is conformal harmonic.

We have the following
\begin{prop}\label{cover-hyp}
Let $\pi : \tilde M \rightarrow M$ be a covering space and let $\tilde g$ be the unique Riemannian metric on the covering manifold $\tilde M$ such that $\pi$ is a Riemannian covering map. Then $(\tilde M, \tilde g)$ is Kobayashi (complete) hyperbolic if and only if $(M,g)$ is Kobayashi (complete) hyperbolic.
\end{prop}

\begin{proof}
By definition $\pi$ is a local isometry, i.e. $\pi^*g=\tilde{g}$. If $u: \mathbb D \rightarrow (\tilde M, \tilde g)$ is a conformal harmonic immersed disc, then $\pi \circ u : \mathbb D \rightarrow (M,g)$ is also a conformal harmonic immersed disc and
$\vert du(0) \cdot e_1\vert_{\tilde{g}} = \vert d(\pi \circ u)(0) \cdot e_1\vert_{g}$.
%$\vert(\partial u/\partial x)(0)\vert_{\tilde{g}} = \vert(\partial (\pi \circ u)/\partial x)(0)\vert_{g}$.
 %$$
%\tilde g(u(0);(\partial u/\partial x)(0),(\partial u/\partial x)(0)) =  g(\pi \circ u(0);(\partial (\pi \circ u)/\partial x)(0),(\partial (\pi \circ u)/\partial x)(0)).
%$$
Conversely, let $v : \mathbb D \rightarrow (M,g)$ be a conformal harmonic immersed disc. Since $\mathbb D$ is simply connected, then for every $y \in \pi^{-1}(v(0))$ there is a unique map $\tilde v : \mathbb D \rightarrow \tilde M$ such that $\tilde v (0) = y$ and $\pi \circ \tilde v = v$. By definition of $\tilde{v}$ and the properties of $\pi$, the energy of $v$ is equal to the energy of $\tilde{v}$. Hence $\tilde v$ is a conformal harmonic disc in $(\tilde M, \tilde g)$ and $\vert d\tilde{v}(0) \cdot e_1\vert_{\tilde{g}} = \vert dv(0) \cdot e_1\vert_{g}$.
%Since $\tilde M$ is a universal covering, the Deck transformations act transitively on the fibers: for every $y' \in \pi^{-1}(v(0))$ there exists a Deck transformation $\gamma \in Isom(\tilde M, \tilde g)$ such that $\gamma(y) = y'$ and $\tilde w:=\gamma \circ \tilde v$ is the unique conformal disc in $(\tilde M, \tilde g)$ such that $\tilde w(0) = y'$.
%\tilde g(y';(\partial \tilde w/\partial x)(0),(\partial \tilde w/\partial x)(0)) =
%$$
%\tilde g(\tilde{v}(0);(\partial \tilde v/\partial x)(0),(\partial \tilde v/\partial x)(0)) =  g(v(0);(\partial v/\partial x)(0),(\partial v/\partial x)(0)).
%$$
We proved that $\pi^*(F_M) = F_{\tilde M}$, implying that $(M,g)$ is Kobayashi hyperbolic if and only if $(\tilde M, \tilde g)$ is. The completeness is proved using chains of conformal harmonic immersed discs, repeating the proof of Theorem~3.2.8 in \cite{Ko}.
\end{proof}

Since the universal cover of a complete Riemannian (resp. K\"ahler) manifold $(M,g)$, with constant sectional curvature $\kappa(g)$ (resp. constant holomorphic sectional $H(g)$), is either the real (resp. complex) Euclidean space, the sphere or the real (resp. complex) hyperbolic space endowed with their classical metrics (see \cite{Ha, Ig} for the complex case), we have as applications of the previous Theorem and of Proposition~\ref{cover-hyp}:

\begin{example}\label{hyp0-ex}
Let $\mathbb R^n$ be endowed with the Euclidean metric $g_{st}$. Then $F_{(\mathbb R^n,  g^{st})} = 0$. %Moreover, since the covering map is an isometry for the Kobayashi metric (see \cite{Ga-Su}), it follows that
More generally, if $(M,g)$ is a complete Riemannian manifold with $K(g) h 0$, then $F_{(M,g)} = 0$.
\end{example}

\begin{example}\label{kah0-ex}
Let $\mathbb C^n$ be endowed with the Euclidean metric $g_{st}$. Then $F_{(\mathbb C^n,  g_{st})} = 0$. More generally, if $(M,g)$ is a complete K\"ahler manifold of complex dimension $n$ with $H(g) = 0$, then $F_{(M,g)} = 0$ since the universal cover of $(M,g)$ is $(\mathbb C^n,g^{Eucl}_{\mathbb C^n})$, see \cite{Ha, Ig}.
\end{example}

\begin{example}\label{ball0-ex}
Let $\mathbb B^n$ be the unit ball in $\mathbb R^n$ endowed with the Euclidean metric $g_{st}$ of $\mathbb R^n$. Then, according to \cite{Fo-Ka}, 
$$F_{(\mathbb B^n, g_{st})}(x,v) = \sqrt{\frac{|v|^2}{1-|x|^2} + \frac{|x \cdot v|^2}{(1-|x|^2)^2}}$$ for every $x \in \mathbb B^n$, $v \in \mathbb R^n$. Here $|v|$ denotes the Euclidean norm of $v$ and $x \cdot v$ the scalar product of $x$ and $v$ in $\mathbb R^n$.
\end{example}

\begin{example}\label{hypball1-ex}
Let $\mathbb H^n:=\{(x_1,\dots,x_n) \in \mathbb R^n / x_n < 0\}$ endowed with the hyperbolic metric $g_{hyp}$ with $K(g_{hyp}) = -4$.
%Then there exists $\lambda > 0$ such that $K(g^{hyp}_{\mathbb H^n})=g^{hyp}_{\mathbb H^n}$. Indeed, the group of isometries of $\mathbb H^n$ being transitive, it is enough to prove that
Then for every $v \in \mathbb R^n$, we have :
$$
F_{(\mathbb H^n, g_{hyp})}({\bf -1},v) \geq \frac{\sqrt{g_{hyp}({\bf -1}) (v,v)}}{\sqrt{2}} = \frac{\vert v \vert}{\sqrt{2}}.
$$
%Moreover, let $f$ be a linear conformal disc, then it is harmonic for the hyperbolic metric (check) and we have Theorem 2.2. in \cite{Fo-Ka}.
%Moreover, since the covering map is an isometry for the Kobayashi metric, it follows that if $(M,g)$ is a complete Riemannian manifold with $K(g) \equiv -c$, for some $c > 0$, then $F_{(M,g)}$ is non denegerate.
More generally, if $(M,g)$ is a complete Riemannian manifold with $K(g) = -c<0$, then $M$ is complete Kobayashi hyperbolic since it is covered by $\mathbb H^n$ by the Killing-Hopf theorem (see \cite{Le}, p.348).

\end{example}

\begin{example}\label{bergm-ex}
Let $b_{\mathbb B^n}$ denote the Bergman metric on the unit ball $\mathbb B^n$ in $\mathbb C^n$, with $H(b_{\mathbb B^n}) = -4/(n+1)$. Then for every orthonormal $v,w \in \mathbb C^n$, the sectional curvature $K(v,w)$ of the real two plane $Span_{\mathbb R}(v,w)$ satisfies $K(v,w)=-\frac{1}{n+1}(1+3<v,iw>^2) \leq -\frac{1}{n+1}$. Hence for every $v \in \mathbb C^n$, we have :
$$
F_{(\mathbb B^n, b_{\mathbb B^n})}(0,v) \geq \frac{\vert v \vert}{\sqrt{8(n+1)}}.
$$
More generally, if $(M,g)$ is a complete K\"ahler manifold with $H(g) = -c < 0$, then $M$ is a complete Kobayashi hyperbolic since the universal cover of $(M,g)$ is $(\mathbb B^n,b_{\mathbb B^n})$, see \cite{Ha, Ig}.

%It follows that if $(M,g)$ is a complete Riemannian manifold with $K(g) \equiv c < 0$, then $M$ is a complete (Kobayashi) hyperbolic.

%(Given a normal vector $v \in \mathbb C^n$, we need to construct a conformal disc $u_{v,w} : \mathbb D \rightarrow \mathbb B^n$ such that $u_{v,w}(0)=0$, $du_{v,w}(0) \cdot e_1 = 4 \sqrt{2 (n+1)}$.)
\end{example}

%-----------------------------------------------------

%-----------------------------------------------------

\end{document}